\documentclass[11pt]{article}

\usepackage[latin1]{inputenc}

\usepackage{amsfonts,amsmath,amssymb,amsthm,graphicx,epsfig,float}
\usepackage[T1]{fontenc}

\usepackage[english,francais]{babel}

{
  \newtheorem{theoreme}{Th\'eor\`eme}
  \newtheorem*{theoreme*}{Th\'eor\`eme}
  \newtheorem{lemme}[theoreme]{Lemme}

  \newtheorem{proposition}[theoreme]{Proposition}
\newtheorem*{corollaire*}{Corollaire}
\newtheorem*{proposition*}{Proposition}
\theoremstyle{remark}
  \newtheorem*{remarque*}{Remarque}
}

\newcounter{ex}

\newenvironment{rem*}{
  \noindent\textbf{Remarque. }}{}

\newcommand{\Cc}{\mathbb{C}}
\newcommand{\Nn}{\mathbb{N}}

\newcommand{\Zz}{\mathbb{Z}}
\newcommand{\Pp}{\mathbb{P}}

\newcommand{\Mcal}{\mathcal{M}}

\newcommand{\Pcal}{\mathcal{P}}

\title{{\bf Un théorème de semi-continuité pour l'entropie des applications méromorphes}}
\author{Henry de Thélin}
\date{}

\begin{document}
\maketitle


\def\figurename{{Fig.}}%
\def\proofname{Preuve}
\def\contentsname{Sommaire}%

\begin{abstract}

Nous montrons un théorème de semi-continuité supérieure pour l'entropie métrique des applications méromorphes.

\end{abstract}

\selectlanguage{english}
\begin{center}
{\bf{ }}
\end{center}

\begin{abstract}

We prove a theorem of uppersemicontinuity for the metric entropy of meromorphic maps.

\end{abstract}

\selectlanguage{francais}

Mots-clefs: dynamique complexe, entropie.

Classification: 32U40, 32H50.

\section*{{\bf Introduction}}
\par

Soit $X$ une variété $C^{\infty}$ et $f: X \longrightarrow X$ une application $C^{\infty}$. On notera $\mathcal{M}$ l'ensemble des probabilités de $X$ invariantes par $f$. Pour $\mu \in \mathcal{M}$, $h_{\mu}(f)$ désignera son entropie métrique.

Newhouse a montré dans ce contexte (voir \cite{N1}) que l'application $\mu \longrightarrow h_{\mu}(f)$ est semi-continue supérieurement. Pour montrer ce résultat, il définit une notion d'entropie locale qui mesure en quelque sorte le défaut de semi-continuité supérieure, puis il montre que cette entropie est nulle en utilisant le théorème de Yomdin (voir \cite{Y}).

L'objectif de cet article est de généraliser ce résultat pour les applications méromorphes.

Soit donc $X$ une variété complexe compacte de dimension $k$ et $f: X \longrightarrow X$ une application méromorphe dominante. On notera $I$ l'ensemble d'indétermination de $f$.
 
Lorsqu'une probabilité $\mu$ ne charge pas $I$, on peut définir le poussé en avant $f_* \mu$, et on dira que $\mu$ est invariante par $f$ si $f_* \mu=\mu$. On notera $\mathcal{M}$ l'ensemble de ces mesures. 

Pour $\mu$ dans $\mathcal{M}$, on peut définir son entropie métrique $h_{\mu}(f)$. Le but de cet article est de donner des résultats de semi-continuité sur ces entropies métriques. L'application $f$ est holomorphe (donc $C^{\infty}$) en dehors de $I$: la difficulté viendra donc de la présence de l'ensemble d'indétermination $I$ sur lequel $f$ n'est même pas une application. Pour résoudre cette difficulté, l'idée est d'ajouter des hypothèses d'intégrabilité de la fonction $\log dist(x,I)$ pour les mesures que l'on considère. Cela permet d'avoir des estimées sur le comportement de $f$ près de $I$.

Plus précisément, le théorème principal est

\begin{theoreme}{\label{th}}

Soit $(\mu_n)$ une suite de $\mathcal{M}$ qui converge vers $\mu$ telle que

$$(H) \mbox{  :   }  \lim_{n \rightarrow + \infty}
\int \log dist(x,I)  d \mu_{n} (x) = \int \log dist(x,I)  d \mu(x) > - \infty.$$

Alors 

$$\limsup_{n \rightarrow + \infty} h_{\mu_n}(f) \leq h_{\mu}(f).$$

\end{theoreme}

Lorsque $f$ est holomorphe (c'est-à-dire $I= \emptyset$), on prend $dist(x,I)$ identiquement égal à $1$ dans l'énoncé de sorte que la condition $(H)$ se réduit au fait que la suite $(\mu_n)$  converge vers $\mu$.

L'idée pour démontrer le théorème est de construire des partitions adaptées aux applications méromorphes (comme dans \cite{dTV}). Ensuite on définira une notion d'entropie locale et enfin en utilisant la généralisation du théorème de Yomdin faite dans \cite{dTV}, nous verrons que cette entropie locale est nulle et nous obtiendrons le théorème.

\section{Partitions adaptées aux applications méromorphes et entropie métrique}

Dans un premier temps, nous rappelons la construction de partitions qui permettent de traiter l'entropie des applications méromorphes (voir \cite{dTV}).

Soit $\eta$ une fonction continue définie sur $X$ comprise entre $0$ et $1$ et qui vaut $0$ sur l'ensemble d'indétermination $I$ de $f$. On a (voir le lemme 2 de \cite{M} et le paragraphe 2.3.4 de \cite{dTV}):

\begin{proposition}{\label{partition}}(Ma{\~n}é)\\
On peut construire une partition dénombrable $\Pcal$ de $X \setminus
\{ \eta = 0 \}$ telle que:
\begin{enumerate}
\item Si  $x \in X \setminus \{ \eta =0 \}$, alors $\mbox{diam} \ \Pcal(x) <
   \eta(y)$ pour tout $y \in \Pcal(x)$ (ici $\Pcal(x)$ est l'atome de la partition qui contient $x$).

\item Pour toute probabilité $\nu$ telle que  $\int \log \eta(x) d \nu(x) > -
   \infty$, on  a $H_{\nu}(\Pcal) < + \infty$. Ici $H_{\nu}(\Pcal)$
désigne l'entropie métrique de la partition $\Pcal$ pour la mesure $\nu$.
\end{enumerate}
\end{proposition}

Rappelons la forme des éléments de $\Pcal$ car cela nous sera utile.

Tout d'abord, il existe $C>0$ et $r_0>0$ telles que pour
 $0 < r \leq r_0$, on a une partition $\Pcal_r$ de $X$ dont les éléments ont un diamètre inférieur à $ r$ et telle que le nombre d'éléments de la partition $| \Pcal_r |$ est plus petit que $C(1/r)^{2k}$.

Maintenant, on définit $V_n:=\{x \mbox{, } e^{-(n+1)} < \eta(x) \leq
  e^{-n} \}$ pour $n \geq 0$. Comme la fonction $\eta$ est comprise entre $0$ et $1$, on a $X
\setminus \{ \eta = 0 \} = \cup_{n \geq 0} V_n$.

Soit $\Pcal$ la partition définie de la façon suivante: pour $n$ fixé,
on considère les ensembles $Q \cap V_n$ pour $Q \in \Pcal_{r_n}$ avec
$r_n = e^{-(n+1)}$. Cela définit une partition de $V_n$. Maintenant,
on obtient la partition $\Pcal$ de $X \setminus \{ \eta=0 \}$ en prenant tous les $n$ compris entre $0$ et $+ \infty$. On peut voir que cette partition vérifie les deux points de la proposition (voir le lemme 2 de \cite{M} et le paragraphe 2.3.4 de \cite{dTV}).

Dans la suite, on regardera cette partition restreinte à $\Omega=X \setminus \cup_{i \geq 0} f^{-i}(I)$ (que l'on notera toujours $\mathcal{P}$). L'avantage des partitions sur $\Omega$ c'est que les $f^{i}$ sont bien définis. En particulier, on peut définir la partition $f^{-i}(\Pcal)$: ses atomes sont les $f^{-i}(P):=\{ x \in \Omega \mbox{ avec } f^{i}(x) \in P \}$ où les $P$ sont les atomes de $\Pcal$. Comme $f(\Omega) \subset \Omega$, on obtient une partition de $\Omega$. Les mesures $\nu \in \mathcal{M}$ ont une masse $1$ sur $\Omega$: la partie de $X$ que l'on a enlevée est donc négligeable pour elles. 

On a maintenant le résultat de continuité suivant:

\begin{proposition}{\label{continuite}}

Soit $(\mu_n)$ une suite de $\Mcal$ qui converge vers $\mu$ telle que

$$\lim_{n \rightarrow + \infty} \int \log \eta(x) d \mu_n(x)=\int \log \eta(x) d \mu(x) > - \infty.$$

Si $\mu$ ne charge pas les bords de la partition $\Pcal$, alors pour tout $q \geq 1$

$$\lim_{n \rightarrow + \infty} H_{\mu_n} \left(\bigvee_{i=0}^{q-1} f^{-i} \Pcal \right)=H_{\mu}\left(\bigvee_{i=0}^{q-1} f^{-i} \Pcal \right).$$

\end{proposition}

Pour démontrer cette proposition, nous aurons besoin du lemme suivant:

\begin{lemme}\label{estimation_integrale}
Sous les mêmes hypothèses que dans la proposition précédente, pour $i,j=0 , \dots , q-1$, on a 
$$ 0 \geq \int_{ \{ \eta \circ f^{i}  \leq \epsilon \} } \log \eta \circ f^j \mbox{ } d \mu_n \geq - \delta(\epsilon)$$
si $n$ est assez grand. Ici $\delta( \epsilon)$ tend vers $0$
quand $\epsilon$ converge vers $0$.

De plus, on a le même résultat avec $\mu$ à la place de $\mu_n $.

\end{lemme}

\begin{proof}

On fait la preuve pour $\mu_n$. C'est essentiellement la même pour $\mu$.

Tout d'abord montrons le lemme pour $i=j=0, \dots , q-1$. On a grâce à l'invariance de $\mu_n$

$$\int_{ \{ \eta \circ f^{j}  \leq \epsilon \} } \log \eta \circ f^j(x) \mbox{ } d \mu_n(x)=\int_{ \{ \eta \leq \epsilon \} } \log \eta(x) d \mu_n(x).$$

Quitte à prendre $\epsilon$ générique de sorte que $\mu$ ne charge pas $\{ \eta= \epsilon \}$, la dernière intégrale converge vers $\int_{ \{ \eta \leq \epsilon \} } \log \eta(x) d \mu(x)$. En effet, d'une part l'hypothèse implique 

$$\lim_{n \rightarrow + \infty} \int \log \eta(x) d \mu_n(x)=\int \log \eta(x) d \mu(x)$$

et d'autre part, $\int_{ \{ \eta > \epsilon \} } \log \eta(x) d \mu_n(x)$ converge vers $\int_{ \{ \eta > \epsilon \} } \log \eta(x) d \mu(x)$.

Par convergence dominée,  $\int_{ \{ \eta \leq \epsilon \} } \log \eta(x) d \mu(x)$ tend vers $0$ quand $\epsilon$ converge vers $0$. On obtient ainsi l'existence d'une fonction $\delta'(\epsilon)$ avec $\lim_{\epsilon \rightarrow 0} \delta'(\epsilon)=0$ telle que:

$$0 \geq \int_{ \{ \eta \circ f^{j}  \leq \epsilon \} } \log \eta \circ f^j(x) \mbox{ } d \mu_n(x) \geq - \delta'(\epsilon)$$

pour $n$ grand.

Maintenant si $i,j=0, \dots ,q-1$, on coupe l'intégrale

$$  \int_{ \{ \eta \circ f^{i}  \leq \epsilon \} } \log \eta \circ f^j(x) \mbox{ } d \mu_n(x)$$

en deux parties

\begin{equation*}
\begin{split}
&\int_{ \{ \eta \circ f^{i}  \leq \epsilon  \} \cap \{\eta \circ f^{j}  \leq  \delta'(\epsilon) \}  } \log
\eta \circ f^{j}(x)  \mbox{ } d \mu_n(x)\\
&+ \int_{ \{ \eta \circ f^{i}   \leq \epsilon \} \cap \{\eta \circ f^{j}  > \delta'(\epsilon) \} } \log
\eta \circ f^{j}(x)   \mbox{ } d \mu_n(x).
\end{split}
\end{equation*}
Le premier terme est plus grand que
$$\int_{  \{ \eta \circ f^{j}  \leq  \delta'(\epsilon) \}  } \log
\eta \circ f^{j}(x)   \mbox{ } d \mu_n(x)$$
et ce terme est minoré par $- \delta'(\delta'(\epsilon))$ si $n$ est assez grand par la première partie de la preuve. Cette quantité converge vers $0$ quand $\epsilon$ tend vers $0$.

Le second terme est plus grand que (si $\epsilon \leq e^{-1}$ et $\delta'(\epsilon) \leq 1$)
\begin{equation*}
\begin{split}
& \int_{ \{ \eta \circ f^{i} \leq \epsilon  \} } \log \delta'(\epsilon)  \mbox{ } d \mu_n(x)\\
& \geq  - \log \delta'(\epsilon) \int_{ \{ \eta \circ f^{i} \leq \epsilon \} } \log \eta \circ f^{i}(x)  \mbox{ } d \mu_n(x) \geq \delta'(\epsilon) \log \delta'(\epsilon)
\end{split}
\end{equation*}
si $n$ est suffisamment grand en utilisant le début de la preuve. Cette quantité converge aussi vers $0$ quand $\epsilon$ tend vers $0$.

\end{proof}

Passons maintenant à la démonstration de la proposition.

\begin{proof}

Pour un multi-indice $s=(s_0, \dots , s_{q-1}) \in \Nn^q$, on note $W_s= V_{s_0} \cap \dots \cap f^{-q+1} V_{s_{q-1}}$. Lorsque $P \in \bigvee_{i=0}^{q-1} f^{-i} \Pcal$, alors $P$ est dans un des $W_s$ pour un certain multi-indice $s$ par construction de la partition $\Pcal$. De plus, le nombre de $P \in \bigvee_{i=0}^{q-1} f^{-i} \Pcal$ qui sont dans $W_s$ est majoré par $C^q e^{2kq} e^{2k |s|}=C_0^q e^{2k |s|}$ avec $|s|= s_0 + \dots + s_{q-1}$.

Maintenant, pour $A \in \Nn^*$, on a
$$H_{\mu_n} \left(\bigvee_{i=0}^{q-1} f^{-i} \Pcal \right)= \sum_{s \in \Nn^q} \sum_{P \in \bigvee_{i=0}^{q-1} f^{-i} \Pcal \mbox{, } P \subset  W_s} - \mu_n(P) \log \mu_n(P)$$
que l'on divise en:
$$\sum_{|s| \geq A} \sum_{P \in \bigvee_{i=0}^{q-1} f^{-i} \Pcal \mbox{, } P \subset  W_s} - \mu_n(P) \log \mu_n(P) + \sum_{|s| <A} \sum_{P \in \bigvee_{i=0}^{q-1} f^{-i} \Pcal \mbox{, } P \subset  W_s} - \mu_n(P) \log \mu_n(P).$$

On commence avec le premier terme. En utilisant l'inégalité

$$- \sum_{i=1}^{m_0} x_i \log x_i \leq \left( \sum_{i=1}^{m_0} x_i \right)
\left( \log m_0 - \log \sum_{i=1}^{m_0} x_i \right)$$

on a

\begin{equation*}
\begin{split}
&R_1(A)=\sum_{|s| \geq A} \sum_{P \in \bigvee_{i=0}^{q-1} f^{-i} \Pcal \mbox{, } P \subset  W_s} - \mu_n(P) \log \mu_n(P) \\
&\leq \sum_{|s| \geq A}  \mu_n(W_s) ( \log ( \# \{ P \in \bigvee_{i=0}^{q-1} f^{-i} \Pcal \mbox{ , } P \subset W_s \}) - \log \mu_n(W_s))\\
&\leq \log C_0^q \sum_{|s| \geq A}  \mu_n(W_s) + 2k \sum_{|s| \geq A} |s|  \mu_n(W_s) + \sum_{|s| \geq A}  \mu_n(W_s) \log \frac{1}{\mu_n(W_s) }.\\
 \end{split}
\end{equation*}
Maintenant, en utilisant le lemme 2.3.5 de \cite{dTV}, on obtient 
$$R_1(A) \leq C(q) e^{- \frac{A}{2q}} +(\log C_0^q  + 2k +1) \sum_{|s| \geq A} |s|  \mu_n(W_s).$$

Sur $W_s$, on a  $- \sum_{j=0}^{q-1} \log \eta \circ f^j \geq |s|$. Par ailleurs les $W_s$ sont deux à deux disjoints et
$$\cup_{|s| \geq A} W_s \subset \cup_{i=0}^{q-1} \cup_{ \{ s \in \Nn^q \mbox{ , } s_i \geq \frac{A}{q} \}} W_s \subset  \cup_{i=0}^{q-1} \{ \eta \circ f^{i} \leq e^{- \frac{A}{q}} \},$$
d'où,
\begin{equation*}
\begin{split}
&R_1(A) \leq C(q) e^{- \frac{A}{2q}} +(\log C_0^q  + 2k +1) \sum_{j=0}^{q-1} \sum_{|s| \geq A} \int_{W_s} -  \log \eta \circ f^j(x) d \mu_n(x)\\
&\leq C(q) e^{- \frac{A}{2q}} +(\log C_0^q  + 2k +1) \sum_{j=0}^{q-1} \sum_{i=0}^{q-1} \int_{  \{ \eta \circ f^{i}(x) \leq e^{- \frac{A}{q}} \} } -  \log \eta \circ f^j(x) d \mu_n(x)\\
 \end{split}
\end{equation*}
qui est aussi petit que l'on veut si on prend $A$ grand puis $n$ grand par le lemme précédent. La même chose est vraie si on remplace $\mu_n$ par $\mu$ dans $R_1(A)$: c'est aussi petit que l'on veut si on prend $A$ grand par le lemme précédent.

On considère maintenant le second terme
$$ \sum_{|s| <A} \sum_{P \in \bigvee_{i=0}^{q-1} f^{-i} \Pcal \mbox{, } P \subset  W_s} - \mu_n(P) \log \mu_n(P).$$
Comme $\mu$ n'a pas de masse sur les bords de la partition cela converge vers
$$ \sum_{|s| <A} \sum_{P \in \bigvee_{i=0}^{q-1} f^{-i} \Pcal \mbox{, } P \subset  W_s} - \mu(P) \log \mu(P),$$
quand $n$ tend vers l'infini car il n'y a qu'un nombre fini d'éléments et $\mu_n$ converge vers $\mu$.

En conclusion, on a bien que $H_{\mu_n}(\bigvee_{i=0}^{q-1} f^{-i} \Pcal)$ qui converge
vers $H_{\mu}(\bigvee_{i=0}^{q-1} f^{-i} \Pcal)$ quand $n$ tend vers l'infini.

\end{proof}

De cette proposition, on en déduit la semi-continuité suivante:

\begin{proposition}{\label{semicontinue}}

Soit $(\mu_n)$ une suite de $\Mcal$ qui converge vers $\mu$ telle que

$$\lim_{n \rightarrow + \infty} \int \log \eta(x) d \mu_n(x)=\int \log \eta(x) d \mu(x) > - \infty.$$

Si $\mu$ ne charge pas les bords de la partition $\Pcal$, alors 

$$\limsup_{n \rightarrow \infty} h_{\mu_n}(f, \Pcal) \leq h_{\mu}(f, \Pcal).$$

\end{proposition}

\begin{proof}

Soit $\delta > 0$.

On choisit d'abord $q \geq 1$ tel que 

$$\left| h_{\mu}(f, \Pcal)- \frac{1}{q} H_{\mu} \left( \bigvee_{i=0}^{q-1} f^{-i}( \Pcal ) \right) \right| \leq \delta.$$

Ensuite, en utilisant la proposition précédente, on a l'existence de $n_0$ avec pour $n \geq n_0$

$$\left| \frac{1}{q} H_{\mu} \left( \bigvee_{i=0}^{q-1} f^{-i}( \Pcal ) \right) - \frac{1}{q} H_{\mu_n} \left( \bigvee_{i=0}^{q-1} f^{-i}( \Pcal ) \right) \right| \leq \delta.$$

Comme

$$h_{\mu_n}(f, \Pcal)= \inf_{q \geq 1} \frac{1}{q} H_{\mu_n} \left( \bigvee_{i=0}^{q-1} f^{-i}( \Pcal ) \right) ,$$

on en déduit que pour $n \geq n_0$

\begin{equation*}
\begin{split}
h_{\mu_n}(f, \Pcal) &\leq \frac{1}{q} H_{\mu_n} \left( \bigvee_{i=0}^{q-1} f^{-i}( \Pcal ) \right) \\
& \leq \frac{1}{q} H_{\mu} \left( \bigvee_{i=0}^{q-1} f^{-i}( \Pcal ) \right) + \delta \leq h_{\mu}(f, \Pcal) + 2 \delta.\\
\end{split}
\end{equation*}

Cela démontre la proposition.

\end{proof}

\section{Un théorème d'entropie locale à la Newhouse}{\label{Newhouse}}

Dans cette partie, on considère toujours une fonction $\eta$ continue définie sur $X$ comprise entre $0$ et $1$ et qui vaut $0$ sur l'ensemble d'indétermination $I$. On note $\Pcal$ la partition associée construite au paragraphe précédent. On rappelle que l'on considèrere cette partition restreinte à $\Omega=X \setminus \cup_{i \geq 0} f^{-i}(I)$ (que l'on notera toujours $\mathcal{P}$).

Le but de ce paragraphe va être de donner un théorème d'entropie locale à la Newhouse (voir \cite{N1}). Commençons par adapter sa définition à notre contexte.

Considérons comme dans \cite{M} les boules dynamiques associée à la fonction $\eta$: pour $x \in \Omega$,

$$B(x,\eta,n,f):=\{ y \in \Omega \mbox{  ,  } dist(f^{i}(x), f^{i}(y)) \leq \eta(f^{i}(x)) \mbox{  ,  } i=0, \cdots , n-1 \}.$$

Pour $\Lambda$ un ensemble de $X$, notons $r(n,\gamma,\eta, \Lambda,x)$ le cardinal maximal d'un ensemble $(n,\gamma)$-séparé inclus dans $B(x,\eta,n,f) \cap \Lambda$. Si $\nu \in \mathcal{M}$, on définit l'entropie locale

$$h_{\nu, \eta, loc}(f):= \lim_{\sigma \rightarrow 1^-} \inf_{ \{ \Lambda \mbox{, } \nu(\Lambda) \geq \sigma \} } \lim_{\gamma \rightarrow 0} \limsup_{n \rightarrow + \infty} \frac{1}{n} \int_{\Lambda} \log r(n, \gamma, \eta, \Lambda, x) d \nu(x).$$

L'entropie locale est plus petite que l'entropie topologique de $f$. En particulier, si $X$ est kählérienne elle est majorée par $\max_{l=0, \cdots, k} \log d_l$ où les $d_l$ sont les degrés dynamiques de $f$ (voir \cite{Gr}, \cite{DS1} et \cite{DS2}).

Revenons au cas général où $X$ est une variété complexe compacte de dimension $k$. On a  alors

\begin{theoreme}{\label{local}}

Pour tout $\nu \in \mathcal{M}$ avec $\int \log \eta(x) d \nu(x) > - \infty$ on a 

$$h_{\nu}(f) \leq h_{\nu}(f, \Pcal) + h_{\nu, \eta, loc}(f).$$

\end{theoreme}

\begin{proof}

Il s'agit ici d'adapter la démonstration de Newhouse (voir \cite{N1}).

Soient $N \in \Nn^*$ et $\alpha= \{E_1, \cdots , E_r , X \setminus \cup_{i=1}^r E_i \}$ une partition de $X$ avec les $E_i$ compacts. On notera encore $\alpha$ cette partition restreinte à $\Omega$. On prend $\sigma < 1$ proche de $1$, $\Lambda$ un ensemble tel que $\nu(\Lambda) \geq \sigma$ puis $\gamma >0$ suffisamment petit pour que $2 \gamma$ soit strictement plus petit que toutes les distances entre $E_i$ et $E_j$ pour $i \neq j$.

On considère dans la suite $n=qN$ et $\beta=(\Lambda \cap \Omega, \Omega \setminus \Lambda )$. Pour simplifier les expressions nous noterons $\Pcal^n = \bigvee_{i=0}^{n-1} f^{-i} (\Pcal)$ et $\alpha_q= \bigvee_{i=0}^{q-1} f^{-iN} (\alpha)$. On a 

$$H_{\nu}(  \alpha_q | \Pcal^n)  \leq H_{\nu}(\alpha_q \vee \beta | \Pcal^n)= H_{\nu}(\beta | \Pcal^n) + H_{\nu}( \alpha_q |  \Pcal^n \vee \beta) \leq \log 2 + H_{\nu}(\alpha_q | \Pcal^n \vee \beta).$$

Estimons maintenant $H_{\nu}(\alpha_q | \Pcal^n \vee \beta)$. On a 

$$H_{\nu}( \alpha_q | \Pcal^n \vee \beta) = \sum_{B \in \Pcal^n \vee \beta} \nu(B) \left( - \sum_{A \in \alpha_q} \nu(A|B) \log \nu(A|B) \right).$$

On coupe la somme $\sum_{B \in \Pcal^n \vee \beta}$ en deux parties: $\sum_{B \in \Pcal^n \vee \beta \mbox{, } B \subset \Lambda}$ et $ \sum_{B \in \Pcal^n \vee \beta \mbox{, } B \subset \Lambda^c}$. 

Commençons par traiter la deuxième partie.

\begin{equation*}
\begin{split}
&\sum_{B \in \Pcal^n \vee \beta \mbox{, } B \subset \Lambda^c} \nu(B) \left( - \sum_{A \in \alpha_q} \nu(A|B) \log \nu(A|B) \right) \leq \nu(\Lambda^c) \log (\mbox{Card}(\alpha_q)) \\
& \leq q  \nu(\Lambda^c) \log (\mbox{Card}(\alpha)) \leq q  (1 - \sigma) \log (\mbox{Card}(\alpha)) .\\
\end{split}
\end{equation*}

Pour l'autre terme

\begin{equation*}
\begin{split}
&\sum_{B \in \Pcal^n \vee \beta \mbox{, } B \subset \Lambda} \nu(B) \left( - \sum_{A \in \alpha_q} \nu(A|B) \log \nu(A|B) \right)\\
&= \int_{\Lambda} - \sum_{A \in \alpha_q} \nu(A|B(x)) \log \nu(A|B(x)) d \nu(x)\\
\end{split}
\end{equation*}

où $B(x)$ est l'élément $B \in \Pcal^n \vee \beta$ qui contient $x$ et où dans la somme à l'intérieur de l'intégrale, on peut se restreindre aux $A \in \alpha_q$ qui rencontrent $B(x)$.

Fixons $x \in \Lambda \cap \Omega$ et majorons le terme dans l'intégrale. 

On note $F(x)$ un ensemble maximal dans $B(x, \eta,n,f) \cap \Lambda$ tel que si $y,z \in F(x)$ avec $y \neq z$ il existe $0 \leq j \leq q-1$ avec $dist(f^{jN}(y), f^{jN}(z)) >  \gamma$. C'est en particulier un ensemble $(\gamma, n )$-séparé donc le cardinal de $F(x)$ est plus petit que $r(n, \gamma, \eta, \Lambda,x)$.

Soit $A \in \alpha_q$ qui rencontre $B(x)$ et $x_A \in A \cap B(x)$. On choisit $\varphi(A) \in F(x)$ tel que $dist(f^{jN}(x_A), f^{jN}(\varphi(A))) \leq  \gamma$ pour $j=0, \cdots , q-1$. Ce $\varphi(A)$ est bien défini. En effet, on a $B(x) \subset B(x, \eta,n,f) \cap \Lambda$ car si $y \in B(x)$, alors $x$ et $y$ sont dans 

$$\Pcal(x) \cap f^{-1}(\Pcal(f(x))) \cap \cdots \cap f^{-n+1}(\Pcal(f^{n-1}(x)))$$

et on conclut en utilisant le fait que le diamètre de $\Pcal(f^{i}(x))$ est plus petit que $\eta(f^{i}(x))$ par la proposition \ref{partition}.

Estimons le nombre de $A$ qui ont le même $\varphi(A)$. On rappelle que $\alpha_q= \bigvee_{i=0}^{q-1} f^{-iN} (\alpha)$. Si $A, A' \in \alpha_q$ avec $A \subset f^{-iN}(E_{l_1})$ et $A' \subset f^{-iN}(E_{l_2})$ où $l_1 \neq l_2$ alors $\varphi(A) \neq \varphi(A')$ car la distance entre $E_{l_1}$ et $E_{l_2}$ est plus grande que $2 \gamma$. Il y a donc au plus $2^q$ ensembles $A$ qui ont le même $\varphi(A)$.

On a alors:

$$- \sum_{A \in \alpha_q} \nu(A|B(x)) \log \nu(A|B(x)) \leq \log (2^q \mbox{Card}(F(x)))$$

c'est-à-dire

$$H_{\nu}( \alpha_q | \Pcal^n \vee \beta) \leq q  (1 - \sigma) \log (\mbox{Card}(\alpha))  + \int_{\Lambda} \log (2^q \mbox{Card}(F(x))) d \nu(x).$$

En combinant ceci avec l'inégalité obtenue au début de la démonstration, on a

$$H_{\nu}( \alpha_q | \Pcal^n ) \leq \log 2 + q (1 - \sigma) \log (\mbox{Card}(\alpha))  + q \log 2 + \int_{\Lambda} \log r(n, \gamma, \eta,\Lambda, x) d \nu(x).$$

Comme on a pris $n=qN$, on obtient 

$$\limsup_{q \rightarrow + \infty} \frac{1}{qN} H_{\nu}( \alpha_q | \Pcal^{qN}) \leq \frac{(1 - \sigma) \log (\mbox{Card}(\alpha))}{N} + \frac{\log 2}{N} + \limsup_{n \rightarrow + \infty} \frac{1}{n} \int_{\Lambda} \log r(n, \gamma, \eta, \Lambda, x) d \nu(x).$$

Maintenant, si on fait tendre $\gamma$ vers $0$, on obtient une inégalité qui est vraie pour tout $\Lambda$ qui vérifie $\nu(\Lambda) \geq \sigma$, d'où:

\begin{equation*}
\begin{split}
\limsup_{q \rightarrow + \infty} \frac{1}{qN} H_{\nu}( \alpha_q | \Pcal^{qN}) \leq & \frac{(1 - \sigma) \log (\mbox{Card}(\alpha))}{N} + \frac{\log 2}{N} \\
&+ \inf_{ \{ \Lambda \mbox{, } \nu(\Lambda) \geq \sigma \} } \lim_{\gamma \rightarrow 0} \limsup_{n \rightarrow + \infty} \frac{1}{n} \int_{\Lambda} \log r(n, \gamma, \eta, \Lambda, x) d \nu(x).\\
\end{split}
\end{equation*}

Enfin, en faisant $\sigma \rightarrow 1$, 

$$\limsup_{q \rightarrow + \infty} \frac{1}{qN} H_{\nu}( \alpha_q | \Pcal^{qN}) \leq  \frac{\log 2}{N} + h_{\nu, \eta, loc}(f).$$

Il reste à voir comment cette inégalité implique celle que l'on cherche.

Le fait que,

$$\lim_{q \rightarrow + \infty} \frac{1}{qN} H_{\nu} \left( \bigvee_{i=0}^{q-1} f^{-iN} (\alpha) \right)= \frac{1}{N} h_{\nu}(\alpha, f^N)$$

et

$$H_{\nu} \left( \bigvee_{i=0}^{q-1} f^{-iN} (\alpha) \right) \leq H_{\nu} ( \Pcal^{qN}) + H_{\nu} \left( \bigvee_{i=0}^{q-1} f^{-iN} (\alpha)  \Big{|}   \Pcal^{qN}  \right)$$

impliquent

$$\frac{1}{N} h_{\nu}(\alpha, f^N) \leq h_{\nu}(f , \Pcal)   +   \frac{\log 2}{N} + h_{\nu, \eta, loc}(f).$$

On utilise ici le fait que l'entropie de $\Pcal$ pour $\nu$ est finie par la proposition \ref{partition} et $\int \log \eta(x) d \nu(x) > - \infty$.

Si on prend le sup sur les partitions $\alpha$ puis que l'on fait tendre $N$ vers l'infini, on obtient le théorème.

\end{proof}

\section{Théorie de Pesin et applications}

Nous allons voir dans ce paragraphe comment la théorie de Pésin permet de contrôler l'entropie locale précédente.

On munit $X$ d'une famille de cartes $(\tau_x)_{x \in X}$ telles que $\tau_x(0)=x$, $\tau_x$ est définie sur une boule $B(0, \epsilon_0) \subset \Cc^k$ avec $\epsilon_0$ indépendant de $x$ et la norme de la dérivée première et seconde de $\tau_x$ sur $B(0, \epsilon_0)$ est majorée par une constante indépendante de $x$. Pour construire ces cartes il suffit de partir d'une famille finie $(U_i, \psi_i)$ de cartes de $X$ et de les composer par des translations.

On considère $\eta$ une fonction continue définie sur $X$ comprise entre $0$ et $1$ et qui vérifie $\eta(x) \leq dist(x,I)$. En particulier, elle vaut $0$ sur l'ensemble d'indétermination $I$. On rappelle que l'on note:

$$B(x,\eta,n,f):=\{ y \in \Omega \mbox{  ,  } dist(f^{i}(x), f^{i}(y)) \leq \eta(f^{i}(x)) \mbox{  ,  } i=0, \cdots , n-1 \}.$$

On a 

\begin{theoreme}{\label{volume}}

Soit $\nu \in \mathcal{M}$ telle que $\int \log \eta(x) d \nu(x) > - \infty$. Alors

$$h_{\nu, \eta, loc}(f) \leq \limsup_{n \rightarrow \infty} \frac{1}{n} \int \log^+ \left( \sup_{Y \mbox{ affine}} \mbox{Vol}(f^n(\tau_x(Y) \cap B(x, 2 \eta,n,f))) \right) d \nu(x).$$

Ici le volume est celui $2l$-dimensionnel réel (compté avec multiplicité) où $l$ est la dimension complexe de $Y$ (qui est un plan affine complexe) et $\log^+(a)=\max(0, \log(a))$ pour $a>0$.

\end{theoreme}

\subsection{\bf{Théorème d'Oseledets et transformée de graphe}}

Dans ce paragraphe, on considère une probabilité $\nu \in \mathcal{M}$ telle que $\int \log \eta(x) d \nu(x) > - \infty$. Cette hypothèse va permettre de définir des exposants de Lyapounov et de faire de la théorie de Pesin. Rappelons que l'on note $\Omega=X \setminus \cup_{i \geq 0} f^{-i}(I)$. Comme $\nu$ ne charge pas $I$, $\nu$ est une probabilité de $\Omega$.

Tout d'abord on définit l'extension naturelle:

$$\widehat{\Omega}:= \{ \widehat{x}=( \cdots, x_0, \cdots , x_n , \cdots) \in \Omega^{\Zz} \mbox{ , } f(x_{n})=x_{n+1} \}.$$

Dans cet espace, $f$ induit une application $\sigma$ qui est le décalage à gauche. Si on note $\pi$ la projection canonique $\pi(\widehat{x})=x_0$ alors $\nu$ se relève en une unique probabilité $\widehat{\nu}$ invariante par $\sigma$ qui vérifie $\pi_{*} \widehat{\nu}=\nu$.

Dans toute la suite, on notera $f_x= \tau_{f(x)}^{-1} \circ f \circ \tau_x$ qui est définie au voisinage de $0$ quand $x$ n'est pas dans $I$. Le cocycle auquel nous allons appliquer la théorie de Pesin est alors donné par:

\begin{equation*}
\begin{split}
A : & \widehat{\Omega} \longrightarrow M_k(\Cc)\\
& \widehat{x} \longrightarrow Df_x(0)\\
\end{split}
\end{equation*}

où $M_k(\Cc)$ est l'ensemble des matrices carrées $k \times k$ à coefficients dans $\Cc$ et $\pi(\widehat{x})=x$. Afin d'avoir un théorème du type Oseledets, nous aurons besoin du lemme suivant:

\begin{lemme}

$$\int \log^+ \| A( \widehat{x}) \| d \widehat{\nu} ( \widehat{x}) < + \infty.$$

\end{lemme}

\begin{proof}

On pose $h(x)= \log^+ \|Df_x(0) \|$. On a alors si $\pi(\widehat{x})=x$,

\begin{equation*}
\begin{split}
&\int \log^+ \| A( \widehat{x}) \| d \widehat{\nu} ( \widehat{x}) = \int  h(x) d \widehat{\nu} ( \widehat{x})\\
&=\int  h \circ \pi ( \widehat{x}) d \widehat{\nu} ( \widehat{x})=\int  h( x) d \nu(x).\\
\end{split}
\end{equation*}

Maintenant comme

$$ \|Df_x(0) \| \leq C' \|Df(x) \| \leq C'' dist(x, I)^{-p}$$

(voir le lemme 2.1 de \cite{DiDu}), le lemme découle de l'intégrabilité de la fonction $\log dist(x,I)$ pour la mesure $\nu$.

\end{proof}

Grâce à ce lemme, on obtient un théorème du type Oseledets (voir  \cite{FLQ} et \cite{Th} ainsi que le théorème 2.3 de \cite{N}, le théorème 6.1 dans \cite{Du} et \cite{M1}):

\begin{theoreme}{\label{pesin}}

 Il existe des réels $\lambda_1(\widehat{x}) > \lambda_2(\widehat{x}) > \cdots > \lambda_{l(\widehat{x})}(\widehat{x}) \geq - \infty$, des entiers $m_1(\widehat{x}), \cdots, m_{l(\widehat{x})}(\widehat{x})$ et un ensemble $\widehat{\Gamma}$ de mesure pleine pour $\widehat{\nu}$ tels que pour $\widehat{x} \in \widehat{\Gamma}$ on ait une décomposition de $\Cc^k$ de la forme $\Cc^k= \bigoplus_{i=1}^{l(\widehat{x})} E_i(\widehat{x})$ où les $E_i(\widehat{x})$ sont des sous-espaces vectoriels de dimension $m_i(\widehat{x})$. Les fonctions $l(\widehat{x})$, $\lambda_i(\widehat{x})$ et $m_i(\widehat{x})$ (pour $i=1 , \cdots , l(\widehat{x})$) sont invariantes par $\sigma$ et on a:

1) $A(\widehat{x}) E_i(\widehat{x}) \subset E_i(\sigma(\widehat{x}))$ avec égalité si $\lambda_i(\widehat{x}) > - \infty$.

2) Pour $v \in E_i(\widehat{x}) \setminus \{0 \}$, on a 
$$\lim_{n \rightarrow + \infty} \frac{1}{n} \log \|A(\sigma^{n-1}(\widehat{x})) \cdots A(\widehat{x}) \| = \lambda_i(\widehat{x}).$$

Si de plus, $\lambda_i(\widehat{x}) > - \infty$, on a la même limite quand $n$ tend vers $- \infty$.

Pour tout $\epsilon > 0$, il existe une fonction $C_{\epsilon} :  \widehat{\Gamma} \longrightarrow GL_k(\Cc)$ telle que pour $\widehat{x} \in \widehat{\Gamma}$:

1) $\lim_{n \rightarrow \infty} \frac{1}{n} \log \| C^{\pm 1}_{\epsilon} (\sigma^n(\widehat{x})) \|=0$ (on parle de fonction tempérée).

2) $C_{\epsilon}(\widehat{x})$ envoie la décomposition standard $\bigoplus_{i=1}^{l(\widehat{x})} \Cc^{m_i(\widehat{x})}$ sur $\bigoplus_{i=1}^{l(\widehat{x})} E_i(\widehat{x})$.

3) La matrice $A_{\epsilon}(\widehat{x})= C_{\epsilon}^{-1}(\sigma(\widehat{x})) A(\widehat{x}) C_{\epsilon}(\widehat{x})$ est diagonale par bloc $(A^1_{\epsilon}(\widehat{x}), \cdots, A^{l(\widehat{x})}_{\epsilon}(\widehat{x}))$ où chaque $A^{i}_{\epsilon}(\widehat{x})$ est une matrice carrée $m_i(\widehat{x}) \times m_i(\widehat{x})$ et
$$\forall v \in \Cc^{m_i(\widehat{x})}   \mbox{    } \mbox{  on a   } \mbox{    } e^{\lambda_i(\widehat{x}) - \epsilon} \|v\| \leq \| A^{i}_{\epsilon}(\widehat{x}) v \| \leq e^{\lambda_i(\widehat{x}) + \epsilon} \|v\|$$

si $\lambda_i(\widehat{x}) > - \infty$ et  

$$\forall v \in \Cc^{m_{l(\widehat{x})}(\widehat{x})}   \mbox{    } \mbox{     } \mbox{    }  \| A^{l(\widehat{x})}_{\epsilon}(\widehat{x}) v \| \leq e^{\epsilon} \|v\|$$

si $\lambda_{l(\widehat{x})}(\widehat{x})= - \infty$.

\end{theoreme}

Notons maintenant $g_{\widehat{x}}$ la lecture de $f_x$ dans les cartes $C_{\epsilon}$ c'est-à-dire $g_{\widehat{x}}= C_{\epsilon}^{-1}(\sigma(\widehat{x}))  \circ f_x \circ C_{\epsilon}(\widehat{x})$ où $\pi(\widehat{x})=x$. On a alors 

\begin{proposition}{\label{proposition10}}

Il existe des constantes $\epsilon_1$ et $C$ qui ne dépendent que de $X$ et $f$ telles que pour $\widehat{x} \in \widehat{\Gamma}$ on ait

1) $g_{\widehat{x}}(0)=0$.

2) $D g_{\widehat{x}}(0)=A_{\epsilon}(\widehat{x})$.

3) Si on note $g_{\widehat{x}}(w)=D g_{\widehat{x}}(0)w + h(w)$, on a 
$$\| Dh(w) \| \leq 2C \|C_{\epsilon}^{-1}(\sigma(\widehat{x}))\|  \|C_{\epsilon}(\widehat{x})\|^2 dist(x,I)^{-p} \|w\|$$
pour $\| w \| \leq \frac{\epsilon_1 dist(x,I)^p}{ \|C_{\epsilon}(\widehat{x})\|}$.

\end{proposition}

\begin{proof}

On suit la preuve de la proposition 8 de \cite{Det1}.

On utilisera (voir le lemme 2.1 de \cite{DiDu})

$$\|Df(x)\| + \|D^2f(x)\| \leq C dist(x,I)^{-p}$$

et on supposera dans la suite que $p \geq 5$.

Commençons par montrer que $g_{\widehat{x}}(w)$ est défini pour $\| w \| \leq \frac{\epsilon_1 dist(x, I)^p}{ \|C_{\epsilon}(\widehat{x})\|}$.

Par construction des cartes $\tau_x$, on peut trouver $\epsilon_0$  qui ne dépend que de $X$ tel que pour $x \in X$ les $\tau_x$ sont définis sur $B(0, \epsilon_0)$ et les $\tau_x^{-1}$ sur $B(x, \epsilon_0)$.

Pour $\| w \| \leq \frac{\epsilon_0}{\|C_{\epsilon}(\widehat{x})\|}$, on a 

\begin{equation*}
\begin{split}
dist(f \circ \tau_x \circ C_{\epsilon}(\widehat{x}) (w), f(x)) &= dist(f \circ \tau_x \circ C_{\epsilon}(\widehat{x}) (w), f \circ \tau_x(0)) \\
&\leq C  \|C_{\epsilon}(\widehat{x}) (w)\| dist(\tau_x([0, C_{\epsilon}(\widehat{x})(w)]), I)^{-p} .\\
\end{split}
\end{equation*}

L'image par $\tau_x$ du segment $[0, C_{\epsilon}(\widehat{x})(w)]$ vit dans la boule $B(x , K \| C_{\epsilon}(\widehat{x})(w) \| )$ où $K$ est une constante qui ne dépend que de $X$. Si $\| w \| \leq \frac{\epsilon_1 dist(x,I)^p}{ \|C_{\epsilon}(\widehat{x})\|} \leq \frac{\epsilon_1 dist(x,I)}{\|C_{\epsilon}(\widehat{x})\|}$, on a alors

$$dist(f \circ \tau_x \circ C_{\epsilon}(\widehat{x}) (w), f(x)) \leq C(1-K \epsilon_1)^{-p} dist(x,I)^{-p} \|C_{\epsilon}(\widehat{x})\| \| w \|.$$

Le dernier terme est plus petit que $\epsilon_0$ si $\|w\| \leq \frac{\epsilon_1 dist(x,I)^p}{  \|C_{\epsilon}(\widehat{x})\|}$ (pour un $\epsilon_1$ qui ne dépend que de $X$ et $f$), ce qui signifie que $g_{\widehat{x}}(w)$ est défini pour de tels $w$.

Passons à la preuve de la proposition.

Le point $1$ est évident et le point $2$ découle du théorème précédent.

Pour le troisième point:

$Dg_{\widehat{x}}(w)=D g_{\widehat{x}}(0) + Dh(w)$, d'où pour $\| w \| \leq \frac{\epsilon_1 dist(x, I)^p}{ \|C_{\epsilon}(\widehat{x})\|}$

\begin{equation*}
\begin{split}
\| Dh(w) \|&= \|Dg_{\widehat{x}}(w)-D g_{\widehat{x}}(0)\|\\
&=\| C_{\epsilon}^{-1}(\sigma(\widehat{x})) \circ Df_x(C_{\epsilon}(\widehat{x}) (w)) \circ C_{\epsilon}(\widehat{x}) - C_{\epsilon}^{-1}(\sigma(\widehat{x})) \circ Df_x(0) \circ C_{\epsilon}(\widehat{x})\|\\
&\leq \| C_{\epsilon}^{-1}(\sigma(\widehat{x}))\|  \|Df_x(C_{\epsilon}(\widehat{x}) (w)) -Df_x(0) \| \| C_{\epsilon}(\widehat{x}) \|. \\
\end{split}
\end{equation*}

Mais on a $\| D^2f(x) \| \leq C d(x,I)^{-p}$ d'où

\begin{equation*}
\begin{split}
\| Dh(w) \| &\leq   C \| C_{\epsilon}^{-1}(\sigma(\widehat{x}))\|    \|C_{\epsilon}(\widehat{x}) (w)\|\| C_{\epsilon}(\widehat{x}) \| dist(\tau_x([0, C_{\epsilon}(\widehat{x})(w)]), I)^{-p} \\
&\leq C(1-K \epsilon_1)^{-p}  \| C_{\epsilon}^{-1}(\sigma(\widehat{x}))\| \|C_{\epsilon}(\widehat{x})\|^2  dist(x,I)^{-p} \|w\|\\
& \leq 2C \|C_{\epsilon}^{-1}(\sigma(\widehat{x}))\|  \|C_{\epsilon}(\widehat{x})\|^2  dist(x,I)^{-p} \|w\|.\\
\end{split}
\end{equation*}

C'est ce que l'on voulait démontrer.

\end{proof}

Un dernier théorème que nous utiliserons est la transformée de graphe dans le cas non-inversible (voir \cite{Du} théorème 6.4). Dans celui-ci $B_l(0,R)$ désigne la boule de centre $0$ et de rayon $R$ dans $\Cc^l$.

\begin{theoreme}(transformée de graphe cas non-inversible){\label{graphe}}(voir théorème 6.4 de \cite{Du})

Soient $A: \Cc^{k_1} \longrightarrow \Cc^{k_1}$, $B: \Cc^{k_2} \longrightarrow \Cc^{k_2}$ des applications linéaires avec $k=k_1 + k_2$. On suppose $A$ inversible, $\|B\| < \|A^{-1}\|^{-1}$ et on note $\xi=1 - \|B\| \|A^{-1}\| \in ]0,1]$. Soient $0 \leq \xi_0 \leq 1$, $0< \delta < 2 \epsilon$ tels que:

$$\xi_0(1- \xi)+2 \delta(1 + \xi_0) \|A^{-1}\| \leq 1  \mbox{   }\mbox{  et }$$

$$(\xi_0 \|B\|+ \delta(1 + \xi_0))(\|A^{-1}\|^{-1}-\delta(1 + \xi_0))^{-1} \leq \xi_0 \mbox{   }\mbox{  et }$$

$$(\|B\| + 2 \delta)e^{-2 \epsilon} + \delta \leq 1 \mbox{   }\mbox{  et } \mbox{   } \delta(1+ \xi_0) \leq \min \{ (\|A^{-1}\|^{-1} - \xi_0 \| B \|)/2 ,\|A^{-1}\|^{-1} -1 \}.$$

Soit $g: B_k(0,R_0) \longrightarrow B_k(0, R_1)$ holomorphe avec $R_0 \leq R_1$, $g(0)=0$, $Dg(0)=(A,B)$ et $\|Dg(w) - Dg(0) \| \leq \delta$ sur $B_k(0,R_0)$. 

Si $\phi : B_{k_2}(0, R) \longrightarrow \Cc^{k_1}$ vérifie $\|\phi(0)\| \leq R'$ et $Lip( \phi) \leq \xi_0$ pour un certain $R' \leq R \leq R_0 /2$ alors il existe $\psi:  B_{k_2} ( 0, R e^{-2\epsilon}) \longrightarrow \Cc^{k_1}$ avec $Lip (\psi) \leq \xi_0$, $\| \psi(0) \| \leq R'$ et $g(\mbox{graphe}(\psi)) \subset \mbox{graphe}(\phi)$.

\end{theoreme}

La seule chose qui change par rapport à \cite{Du}, c'est le fait de mettre $R'$ à la place de $R$ à la fin de l'énoncé. C'est une adaptation immédiate de la preuve de  \cite{Du}.

\subsection{\bf{Démonstration du théorème \ref{volume}}}

Soit $\nu \in \mathcal{M}$ telle que $\int \log \eta(x) d \nu(x) > - \infty$. 

Soit $\delta > 0$. Pour $\sigma$ proche de $1$, on a  que la différence entre $h_{\nu, \eta, loc}(f)$ et

$$\inf_{ \{ \Lambda \mbox{, } \nu(\Lambda) \geq \sigma \} } \lim_{\gamma \rightarrow 0} \limsup_{n \rightarrow + \infty} \frac{1}{n} \int_{\Lambda} \log r(n, \gamma, \eta,\Lambda, x) d \nu(x)$$

est plus petite que $\delta$.

Dans un premier temps nous allons construire un ensemble particulier $\Lambda$ sur lequel nous aurons de bonnes estimées. Ensuite nous majorerons $r(n, \gamma, \eta,\Lambda, x)$ sur cet ensemble.

\medskip

{\bf Construction de $\Lambda$:}

\medskip

On reprend les notations du paragraphe précédent. Tout d'abord, pour $\alpha_0 > 0$ suffisamment petit, la mesure pour $\widehat{\nu}$ de 

$$B_1=\left\{ \widehat{x} \in \widehat{\Gamma} \mbox{ , } \min_{ \{ \lambda_i(\widehat{x}) \neq 0 \}} | \lambda_i(\widehat{x}) | \geq \alpha_0 \right\}$$

est strictement plus grande que $\sigma$. Ensuite, on fixe $\epsilon > 0$ très petit devant $\alpha_0$ et $\delta$ et on considère

$$B_2=\left\{ \widehat{x} \in \widehat{\Gamma} \mbox{ , } \alpha_1 \leq \| C_{\epsilon}(\widehat{x})^{\pm 1} \| \leq \frac{1}{\alpha_1}  \right\}.$$

Si $\alpha_1$ est suffisamment petit on a $\widehat{\nu}(B_1 \cap B_2) > \sigma$. 

Maintenant, si on note

$$\mathcal{A}_{n_0}:= \{x \mbox{ , } \forall n \geq n_0 \mbox{   } \min_{i=0, \cdots, n-1} \eta(f^{i}(x)) \geq e^{- \epsilon n } \}$$

on a 

\begin{lemme}

On a $\nu( \cup_{n_0 \in \Nn} \mathcal{A}_{n_0})=1$.

\end{lemme}

\begin{proof}

Par le théorème de Birkhoff, il existe un ensemble $\mathcal{A}$ de mesure $1$ pour $\nu$ tel que pour $x \in \mathcal{A}$ on ait:

$$\lim_{n \rightarrow + \infty} \frac{1}{n} \sum_{i=0}^{n-1} \log \eta(f^{i}(x)) = \int \log \eta(x) d \nu(x) > - \infty.$$

Classiquement, cela implique que pour $x \in \mathcal{A}$ 

$$\lim_{n \rightarrow + \infty} \frac{1}{n}  \log \eta(f^{n}(x)) =0.$$

Soit $x \in \mathcal{A}$. Il existe $n_1$ tel que $\frac{1}{n}  \log \eta(f^{n}(x))  \geq -\epsilon $ si $n \geq n_1$. On a donc, pour $n > n_1$,

$$\min_{i=n_1, \cdots , n-1} \eta(f^{i}(x)) \geq \min_{i=n_1, \cdots , n-1} e^{- \epsilon  i} \geq e^{- \epsilon  n}.$$

Maintenant, comme $x \in \mathcal{A}$ les $\eta(x), \cdots , \eta(f^{n_1}(x))$ sont non nuls et il existe $n_2$ tel que pour $n \geq n_2$

$$\min_{i=0, \cdots , n_1} \eta(f^{i}(x)) \geq e^{- \epsilon  n}.$$

Le point $x$ est donc dans $\mathcal{A}_{\max(n_1, n_2)}$. Autrement dit $\mathcal{A}$ est inclus dans $\cup_{n_0 \in \Nn} \mathcal{A}_{n_0}$ et le lemme en découle.

\end{proof}

 On déduit de ce lemme que $\nu(\pi(B_1 \cap B_2) \cap \mathcal{A}_{n_0}) \geq \sigma$ si $n_0$ est grand. On notera dans la suite $\Lambda= \pi(B_1 \cap B_2 ) \cap \mathcal{A}_{n_0}$. C'est l'ensemble que l'on voulait construire.

Il s'agit maintenant de majorer 

$$\lim_{\gamma \rightarrow 0} \limsup_{n \rightarrow + \infty} \frac{1}{n} \int_{\Lambda} \log r(n, \gamma, \eta,\Lambda, x) d \nu(x).$$

\medskip

{\bf Majoration de $r(n, \gamma, \eta,\Lambda, x)$ sur $\Lambda$:}

\medskip 

Commençons par prendre $\gamma >0$ tel que la différence entre 

$$\limsup_{n \rightarrow + \infty} \frac{1}{n} \int_{\Lambda} \log r(n, \gamma, \eta,\Lambda, x) d \nu(x)$$

et 

$$\lim_{\gamma \rightarrow 0} \limsup_{n \rightarrow + \infty} \frac{1}{n} \int_{\Lambda} \log r(n, \gamma, \eta,\Lambda, x) d \nu(x)$$

soit plus petite que $\delta $.

La majoration de $r(n, \gamma, \eta,\Lambda, x)$ va se faire en deux étapes. Dans un premier temps nous construisons pour $x \in \Lambda$ un ensemble $Q(x)$ qui sera feuilleté par des variétés stables approchées. Ensuite nous utiliserons ces ensembles pour fabriquer un $Y$ affine tel que le volume de $f^n(\tau_x(Y) \cap B(x, 2 \eta,n,f))$ soit essentiellement plus grand que $r(n, \gamma, \eta,\Lambda, x)$.

\medskip

{\bf Etape $1$: construction des $Q(x)$:}

\medskip

Soit $x \in \Lambda$ et $\widehat{x} \in B_1 \cap B_2$ tel que $\pi( \widehat{x})=x$. On note $E_u(\widehat{x})$ la somme directe de tous les $E_i(\widehat{x})$ correspondant aux exposants strictement positifs et $E_s(\widehat{x})$ la somme directe des $E_i(\widehat{x})$ associés aux exposants négatifs ou nuls. La dimension complexe de $E_u(\widehat{x})$ sera notée $u$. Elle est comprise entre $0$ et $k$.

On se place maintenant dans

$$C_{\epsilon}^{-1}(\sigma^n(\widehat{x})) E_u(\sigma^n(\widehat{x})) \oplus C_{\epsilon}^{-1}(\sigma^n(\widehat{x})) E_s(\sigma^n(\widehat{x}))$$

et on part de 

$$\{a_1\} \times \cdots \times \{a_u \} \times B_2(0, e^{-4 p \epsilon n})$$

où $B_2(0, e^{-4 p \epsilon n})$ est la boule de centre $0$ et de rayon $e^{-4 p \epsilon n}$ dans $\Cc^{k-u}$ et $(a_1, \cdots , a_u) \in B_1(0, e^{-12 p \epsilon n})$ avec $B_1(0, e^{-12 p \epsilon n})$ la boule de centre $0$ et de rayon $e^{-12 p \epsilon n}$ dans $\Cc^{u}$. Cet ensemble est un graphe $(\Phi_n(Y),Y)$ au-dessus d'une partie de $ C_{\epsilon}^{-1}(\sigma^n(\widehat{x})) E_s(\sigma^n(\widehat{x}))$ (avec $\Phi_n(Y)=(a_1, \cdots , a_u)$). Dans le cas où $u=0$ ou $u=k$, tout ceci garde un sens en identifiant $\Cc^k$ avec $\{0\} \times \Cc^k$ ou $\Cc^k \times \{0\}$. Soit $0 < \xi_0 < \frac{1}{2}$ (petit aussi par rapport à $\alpha_1$).

\begin{lemme} 

Il existe un graphe $(\Phi_{n-1}(Y),Y)$ au-dessus de 
$$B_2(0, e^{-4 p \epsilon n - 2 \epsilon} ) \subset C_{\epsilon}^{-1}(\sigma^{n-1}(\widehat{x})) E_s(\sigma^{n-1}(\widehat{x}))$$
avec $\mbox{Lip }\Phi_{n-1} \leq \xi_0$, $g_{\sigma^{n-1}(\widehat{x})}(\mbox{graphe de }\Phi_{n-1}) \subset \mbox{graphe de }\Phi_n$ et $\| \Phi_{n-1}(0) \| \leq e^{-12 p \epsilon n}$.

\end{lemme}

\begin{proof}

Par la proposition \ref{proposition10} et le théorème \ref{pesin}, dans le repère

$$C_{\epsilon}^{-1}(\sigma^{n-1}(\widehat{x})) E_u(\sigma^{n-1}(\widehat{x})) \oplus C_{\epsilon}^{-1}(\sigma^{n-1}(\widehat{x})) E_s(\sigma^{n-1}(\widehat{x})),$$

on peut écrire  $g_{\sigma^{n-1}(\widehat{x})}$ sous la forme

$$g_{\sigma^{n-1}(\widehat{x})}(X,Y)=(A_{n-1} X + R_{n-1}(X,Y), B_{n-1}Y + U_{n-1}(X,Y))$$

avec:

$$D g_{\sigma^{n-1}(\widehat{x})}(0)=A_{\epsilon}(\sigma^{n-1}(\widehat{x}))=(A_{n-1}, B_{n-1})$$

où

$$\| A_{n-1}^{-1} \|^{-1} \geq e^{\alpha_0 - \epsilon} \geq e^{2 \epsilon} \mbox{ et } \|B_{n-1} \| \leq e^{\epsilon}.$$

En particulier, en utilisant les notations du théorème \ref{graphe}

$$1-  \xi= \|B_{n-1} \| \| A_{n-1}^{-1} \| \leq e^{-\epsilon} < 1.$$

Estimons maintenant le $\delta$ du théorème \ref{graphe}. On a par la proposition \ref{proposition10}

$$ \| D g_{\sigma^{n-1}(\widehat{x})}(0) - D g_{\sigma^{n-1}(\widehat{x})}(w) \| \leq 2C  \|C_{\epsilon}^{-1}(\sigma^n(\widehat{x}))\|  \|C_{\epsilon}(\sigma^{n-1}(\widehat{x}))\|^2 dist(f^{n-1}(x),I)^{-p} \|w\|$$

pour $\|w\| \leq \frac{\epsilon_1 dist(f^{n-1}(x),I)^p}{\| C_{\epsilon}(\sigma^{n-1}(\widehat{x})) \|}$. Mais comme les fonctions $\| C_{\epsilon}^{\pm 1} \|$ sont tempérées et que $\widehat{x} \in B_2$, on peut supposer que

$$\|C_{\epsilon}^{-1}(\sigma^n(\widehat{x}))\|  \|C_{\epsilon}(\sigma^{n-1}(\widehat{x}))\|^2  \leq \frac{1}{\alpha_1^3} e^{3 \epsilon n}.$$

Maintenant, comme $x \in \mathcal{A}_{n_0}$ et $\eta(x) \leq dist(x,I)$, on a $dist(f^{n-1}(x),I) \geq \eta(f^{n-1}(x)) \geq e^{-\epsilon n }$ pour $n \geq n_0$ et alors

$$ \| D g_{\sigma^{n-1}(\widehat{x})}(0) - D g_{\sigma^{n-1}(\widehat{x})}(w) \| \leq \frac{2C}{\alpha_1^3} e^{3 \epsilon n} e^{p \epsilon n} \| w \| \leq e^{2p \epsilon n} \| w \| $$

si $\|w\| \leq e^{-2p \epsilon n} \leq \frac{\alpha_1 \epsilon_1 e^{-p \epsilon n }}{ e^{\epsilon n}} \leq \frac{\epsilon_1 dist(f^{n-1}(x),I)^p}{\| C_{\epsilon}(\sigma^{n-1}(\widehat{x})) \|}$ pour $n$ grand.

En particulier, pour $\| w \| \leq R_0= e^{-3p \epsilon n}$ et $n$ grand, le $\delta$ du théorème \ref{graphe} est plus petit que $e^{-p \epsilon n}$.  

Le lemme découle alors de la proposition \ref{graphe} en prenant $R=e^{-4p \epsilon n}$ et $R'=e^{-12p \epsilon n}$ avec $n$ grand. Ici le $n$ grand est indépendant du $x$ choisi dans $\Lambda$.

\end{proof}

Maintenant on recommence ce que l'on vient de faire avec $g_{\sigma^{n-2}(\widehat{x})}$ au lieu de $g_{\sigma^{n-1}(\widehat{x})}$. On se place toujours dans la boule $B(0, R_0)$ avec $R_0= e^{-3p \epsilon n}$ et on prend $R=e^{-4p \epsilon n -2 \epsilon}$.  On obtient ainsi (toujours pour le même $n$ grand que précédemment) l'existence d'un graphe $(\Phi_{n-2}(Y),Y)$ au-dessus de 
$$B_2(0, e^{-4 p \epsilon n - 4 \epsilon} ) \subset C_{\epsilon}^{-1}(\sigma^{n-2}(\widehat{x})) E_s(\sigma^{n-2}(\widehat{x}))$$
avec $\mbox{Lip }\Phi_{n-2} \leq \xi_0$, $g_{\sigma^{n-2}(\widehat{x})}(\mbox{graphe de }\Phi_{n-2}) \subset \mbox{graphe de }\Phi_{n-1}$ et $\| \Phi_{n-2}(0) \| \leq e^{-12 p \epsilon n}$.

On continue ainsi le procédé. A la fin on obtient un graphe $(\Phi_0(Y),Y)$ au-dessus de 
$$B_2(0, e^{-4 p \epsilon n - 2 \epsilon n} ) \subset C_{\epsilon}^{-1}(\widehat{x}) E_s(\widehat{x})$$
avec $\mbox{Lip }\Phi_{0} \leq \xi_0$, $g_{\widehat{x}}(\mbox{graphe de }\Phi_{0}) \subset \mbox{graphe de }\Phi_{1}$ et $\| \Phi_{0}(0) \| \leq e^{-12 p \epsilon n}$.

Grâce aux estimées sur $B_2$, l'image de ce graphe par $ C_{\epsilon}(\widehat{x})$ est un graphe $(\Psi_0(Y),Y)$ dans le repère $ E_u(\widehat{x}) \oplus E_s(\widehat{x})$ au-dessus d'une boule $B_2(0, e^{-4 p \epsilon n - 3 \epsilon n} )$. Il vérifie $\mbox{Lip }\Psi_{0} \leq \frac{\xi_0}{\alpha_1^2}$ qui est aussi petit que l'on veut pourvu que $\xi_0$ le soit par rapport à $\alpha_1$ (ce que l'on a supposé) et $\| \Psi_0(0) \| \leq e^{-12 p \epsilon n + \epsilon n}$.

On notera $Q(x)$ la réunion de tous ces graphes $(\Psi_0(Y),Y)$ quand $(a_1, \cdots , a_u)$ décrit $B_1(0, e^{-12 p \epsilon n})$. 

\medskip

{\bf Etape 2: fin de la majoration:}

\medskip

Soit $x \in \Lambda$ fixé et $x_1, \cdots , x_N$ un ensemble maximal de points $(n, \gamma)$-séparés dans $B(x, \eta, n ,f) \cap \Lambda$. On supposera dans la suite que $N \geq e^{40 pk \epsilon n}$ (sinon on verra à la fin de ce paragraphe que l'on obtient directement la majoration voulue).

Quitte à remplacer $N$ par $N/K_0$ où $K_0$ ne dépend que de $X$, on peut supposer que la dimension des $E_u(\widehat{x_i})$ est la même pour tous les $x_i$ (on la notera $u$) et que tous les $\tau_{x_i}$ sont égaux à une seule carte $\psi:U \longrightarrow X$ modulo des translations. En particulier la distance entre le bord de $U$ et les $\psi^{-1}(x_i)$ est plus grande que $\epsilon_0$.

Toujours quitte à remplacer $N$ par $N/K_0$, on peut supposer que les graphes qui constituent les $\psi^{-1} ( \tau_{x_i}(Q(x_i)))$ (qui sont juste les translatés des $Q(x_i)$) sont des graphes au-dessus d'un plan complexe $P$ de dimension $k-u$. Par ailleurs comme les $\| \Psi_0(0) \|$ sont très petits devant $e^{-4 p \epsilon n - 3 \epsilon n} $, pour chaque $\psi^{-1} ( \tau_{x_i}(Q(x_i)))$, on peut supposer que tous les graphes qui le constituent sont des graphes au-dessus d'une même boule de rayon $e^{-6 p  \epsilon n}$ dans $P$ (quitte à ne garder que cette partie là). Si on note $G(x_i)$ un des graphes qui constitue $\psi^{-1} ( \tau_{x_i}(Q(x_i)))$, le volume $2(k-u)$-dimensionnel réel de la projection de tous ces graphes $G(x_i)$ ($i=1, \cdots, N$) est supérieur à $N e^{-12 p k \epsilon n}$.

Notons $\pi_1$ la projection orthogonale sur $P$. Comme $\pi_1(U)$ vit dans un compact de $P$, on peut trouver un point $a$ de $P$ tel que la fibre $L= \pi_1^{-1}(a)$ coupe $\cup_{i=1}^{N} G(x_i)$ en au moins $N e^{-12 p k \epsilon n}$ points (éventuellement divisé par une constante qui ne dépend que de $U$).  Pour simplifier les notations, nous noterons $x_i$ avec $i= 1 , \cdots , Ne^{-12 p k \epsilon n}$ ces points. Par construction, quand $L$ coupe $G(x_i)$ il coupe tous les graphes qui constituent $\psi^{-1} ( \tau_{x_i}(Q(x_i)))$. 

Signalons deux cas particuliers: quand $u=0$, $L$ est en fait un point commun aux $G(x_i)$ qui eux sont des ellipsoïdes. Lorsque $u=k$, $L$ est égal à $U$ et les $G(x_i)$ sont des points.

Montrons maintenant que les $\psi^{-1} ( \tau_{x_i}(Q(x_i)))$ sont deux à deux disjoints. Si $i \neq j$ les points $x_i$ et $x_j$ sont $(n , \gamma)$-séparés et il existe donc $l$ compris entre $0$ et $n-1$ avec $dist(f^l(x_i), f^l(x_j)) \geq \gamma$. D'autre part, on a 

\begin{lemme}

Pour $l=0, \cdots , n-1$ et $i=1, \cdots ,N$ 

$$f^l(\tau_{x_i}(Q(x_i))) \subset B(f^l(x_i), e^{-2 p \epsilon n}).$$

\end{lemme}

\begin{proof}

Par construction, on a pour $l=0, \cdots, n$

$$g_{\sigma^{l-1}(\widehat{x_i})} \circ \cdots \circ g_{\widehat{x_i}} (C^{-1}_{\epsilon}(\widehat{x_i})(Q(x_i))) \subset B(0, e^{-3 p \epsilon n})$$

(où par convention $g_{\sigma^{l-1}(\widehat{x_i})} \circ \cdots \circ g_{\widehat{x_i}}$ est l'identité si $l=0$).

De plus,

\begin{equation*}
\begin{split}
g_{\sigma^{l-1}(\widehat{x_i})} \circ \cdots \circ g_{\widehat{x_i}} &= C^{-1}_{\epsilon}(\sigma^{l}(\widehat{x_i})) \circ f_{f^{l-1}(x_i)} \circ \cdots \circ f_{x_i} \circ C_{\epsilon}(\widehat{x_i})\\
&= C^{-1}_{\epsilon}(\sigma^{l}(\widehat{x_i})) \circ \tau^{-1}_{f^{l}(x_i)} \circ f^l \circ \tau_{x_i} \circ C_{\epsilon}(\widehat{x_i}).\\
\end{split}
\end{equation*}

On a donc

$$f^l(\tau_{x_i}(Q(x_i))) \subset \tau_{f^{l}(x_i)} \circ C_{\epsilon}(\sigma^{l}(\widehat{x_i})) (B(0, e^{-3 p \epsilon n}))$$

ce qui donne le lemme (toujours pour $n$ grand) grâce au contrôle de $\| C_{\epsilon}(\widehat{x_i}) \|$ par $1/ \alpha_1$, le fait que $C_{\epsilon}$ est tempérée et le contrôle des dérivées premières de $\tau_x$ pour $x \in X$.

\end{proof}

Le fait que $dist(f^l(x_i), f^l(x_j)) \geq \gamma$ combiné au lemme implique que les $\tau_{x_i}(Q(x_i))$ sont deux à deux disjoints (et donc les $\psi^{-1} ( \tau_{x_i}(Q(x_i)))$ aussi).

Dans le cas particulier où $u=0$ cela implique que $N e^{-12pk \epsilon n}$ est inférieur à $1$ ce qui contredit le fait que $N \geq e^{40 pk \epsilon n}$. Dans la suite, on supposera donc que $u > 0$.

Pour $i= 1 , \cdots , Ne^{-12 p k \epsilon n}$, l'ensemble $C^{-1}_{\epsilon}(\sigma^{n}(\widehat{x_i})) \circ \tau^{-1}_{f^n(x_i)} (f^n(\psi(L) \cap \tau_{x_i}(Q(x_i))))$ rencontre par construction tous les 

$$\{a_1\} \times \cdots \times \{a_u \} \times B_2(0, e^{-4 p \epsilon n})$$

où $(a_1, \cdots , a_u) \in B_1(0, e^{-12 p \epsilon n})$ (car $L$ coupe tous les graphes qui constituent $\psi^{-1} ( \tau_{x_i}(Q(x_i)))$). En particulier le volume $2u$-dimensionnel de $C^{-1}_{\epsilon}(\sigma^{n}(\widehat{x_i})) \circ \tau^{-1}_{f^n(x_i)} (f^n(\psi(L) \cap \tau_{x_i}(Q(x_i))))$ est minoré par $e^{-24 p k \epsilon n}$. Comme les $x_i$ sont dans $\Lambda$, que les $C_{\epsilon}^{\pm 1}$ sont tempérées, et que les dérivées des $\tau_y^{\pm 1}$ sont bornées, on obtient que le volume de $f^n(\psi(L) \cap \tau_{x_i}(Q(x_i)))$ est plus grand que $e^{-25 p k \epsilon n}$. Par suite celui de $f^n(\psi(L) \cap (\cup_{i=1}^{N e^{-12pk \epsilon n}} \tau_{x_i}(Q(x_i))))$ compté avec multiplicité est minoré par $N e^{-40 p k \epsilon n}$ car les $\tau_{x_i}(Q(x_i))$ sont disjoints.

Maintenant, on a $\cup_{i=1}^{N e^{-12pk \epsilon n}} \tau_{x_i}(Q(x_i)) \subset B(x, 2 \eta, n,f)$. En effet, soit $y \in \cup_{i=1}^{N e^{-12pk \epsilon n}} \tau_{x_i}(Q(x_i))$ et $i$ tel que $y \in \tau_{x_i}(Q(x_i))$. Comme $x_i \in B(x,  \eta, n,f)$, on a par le lemme précédent, pour $l=0, \cdots, n-1$

\begin{equation*}
\begin{split}
dist(f^l(y), f^l(x)) &\leq dist(f^l(y), f^l(x_i)) + dist(f^l(x_i), f^l(x))\\
& \leq e^{-2p \epsilon n} + \eta(f^l(x)) \leq 2 \eta(f^l(x)).
\end{split}
\end{equation*}

La dernière inégalité provient du fait que $x \in \mathcal{A}_{n_0}$.

En combinant ce qui précède on a donc que le volume $2u$-dimensionnel de $f^n(\psi(L) \cap B(x, 2 \eta , n ,f))$ est minoré par $N e^{-40 p k \epsilon n}$. L'application $\tau_x$ s'écrit $\psi \circ t$ où $t$ est une translation, donc en posant $Y= t^{-1}(L)$, on a

$$\mbox{Vol}(f^n(\tau_x(Y) \cap B(x, 2 \eta, n ,f))) \geq  N e^{-40 p k \epsilon n}$$

avec $Y$ affine. Autrement dit

$$N=r(n, \gamma, \eta, \Lambda ,x) \leq \mbox{Vol}(f^n(\tau_x(Y) \cap B(x, 2 \eta, n ,f))) e^{40 p k \epsilon n}.$$

Dans le cas où $N < e^{40 p k \epsilon n}$, on a 

$$N=r(n, \gamma, \eta, \Lambda ,x) \leq max(\mbox{Vol}(f^n(\tau_x(Y) \cap B(x, 2 \eta, n ,f))),1) e^{40 p k \epsilon n}.$$

D'où finalement,

\begin{equation*}
\begin{split}
h_{\nu, \eta, loc}(f) - \delta & \leq \inf_{ \{ \Lambda \mbox{, } \nu(\Lambda) \geq \sigma \} } \lim_{\gamma \rightarrow 0} \limsup_{n \rightarrow + \infty} \frac{1}{n} \int_{\Lambda} \log r(n, \gamma, \eta,\Lambda, x) d \nu(x)\\
&\leq \limsup_{n \rightarrow + \infty} \frac{1}{n} \int_{\Lambda} \log r(n, \gamma, \eta,\Lambda, x) d \nu(x) + \delta\\
& \leq \limsup_{n \rightarrow + \infty} \frac{1}{n} \int_{\Lambda} \log^+ \left( \sup_{Y \mbox{affine}} \mbox{Vol}(f^n(\tau_x(Y) \cap B(x, 2 \eta,n,f))) \right) d \nu(x) + \delta + 40pk \epsilon.\\
\end{split}
\end{equation*}

C'est ce que l'on voulait démontrer.

\section{Démonstration du théorème}

Soit $m \in \Nn^*$. On définit comme dans \cite{dTV}

$$\rho_m(x)= \left( \frac{dist(x,I) \times \cdots \times dist(f^{m-1}(x), I)}{K^m} \right)^p.$$

La fonction $\rho_m$ vaut $0$ sur $I$, on peut donc définir pour $\nu \in \mathcal{M}$ la quantité $h_{\nu, \frac{\rho_m}{2} , loc}(f)$ comme au paragraphe \ref{Newhouse}. On a  alors

\begin{theoreme}

Soit $L >0$. Pour tout $\delta > 0$, il existe $m_0$ tel que pour $m \geq m_0$

$$\sup_{ \{ \nu \in \mathcal{M} \mbox{,} \int \log d(x,I) d \nu(x) \geq -L \} } h_{\nu, \frac{\rho_m}{2} , loc}(f) \leq \delta.$$

\end{theoreme}

\begin{proof}

La démonstration de ce théorème repose sur l'inégalité volumique du paragraphe précédent ainsi que la version méromorphe du théorème de Yomdin développée dans la proposition 2.3.2 de \cite{dTV}. On adoptera ici les notations de cette proposition.

Fixons $L>0$ puis $\delta >0$. On prend $r$ tel que $\frac{1}{r} \log K < \delta$ et $\frac{L}{r} < \delta$. Soit $m_0$ tel que $\max_{l=0, \cdots , 2k} \left( \frac{1}{m} \log(C(X,l,r)) \right) < \delta$ pour $m \geq m_0$. On considère $\nu \in \mathcal{M}$ telle que $\int \log d(x,I) d \nu(x) \geq -L$.

Comme $\nu$ est invariante par $f$, la fonction $\log  \frac{\rho_m}{2}$ est intégrable. Si on veut on a $dist(x,I) \geq \frac{\rho_m(x)}{2}$ ce qui implique par le théorème \ref{volume} que

$$h_{\nu, \frac{\rho_m}{2}, loc}(f) \leq \limsup_{n \rightarrow \infty} \frac{1}{n} \int \log^{+} \left( \sup_{Y \mbox{affine}} \mbox{Vol}(f^n(\tau_x(Y) \cap B(x,  \rho_m,n,f) )) \right) d \nu(x).$$

On voit que les boules $B(x,  \rho_m,n,f)$ sont incluses dans celle $B_n(x)$ définies au début du paragraphe 2.3 de \cite{dTV}. En particulier, par la proposition 2.3.2 de \cite{dTV}, le volume de $f^n(\tau_x(Y) \cap B(x,  \rho_m,n,f))$ est majoré par (ici $l$ est la dimension réelle de $Y$)

$$C(X,l,r)^{n/m+2m} \times K^{\frac{2npl}{r} + \frac{4mpl}{r}} \times \prod_{0 \leq i \leq n-1} dist(f^{i}(x),I)^{\frac{-4pl}{r}}.$$

Si on pose $C(X,r)= \max_{l=0, \cdots , 2k} C(X,l,r)$, on a donc

\begin{equation*}
\begin{split}
& \frac{1}{n} \int \log^{+} \left( \sup_{Y \mbox{affine}} \mbox{Vol}(f^n(\tau_x(Y) \cap B(x, \rho_m,n,f))) \right) d \nu(x)\\
&\leq \frac{1}{m} \log ( C(X,r)) + \frac{2pl}{r} \log K - \frac{4kp}{rn} \sum_{i=0}^{n-1} \int \log dist(f^{i}(x),I) d \nu(x) + \frac{C(X,m,r,K,p)}{n}\\
&\leq \delta + 2kp \delta - \frac{4kp}{r} \int \log dist(x,I) d \nu(x) + \frac{C(X,m,r,K,p)}{n}\\
&\leq \delta + 2kp \delta + 4kp \delta + \frac{C(X,m,r,K,p)}{n}.\\
\end{split}
\end{equation*}

On obtient ainsi

$$h_{\nu, \frac{\rho_m}{2}, loc}(f) \leq 7 kp \delta$$

pour tout $\nu \in \mathcal{M}$ telle que $\int \log d(x,I) d \nu(x) \geq -L$.

C'est ce que l'on voulait démontrer.

\end{proof}

Passons maintenant à la démonstration du théorème principal.

\begin{proof}{Démonstration du théorème \ref{th}:}

Soit $(\mu_n)$ une suite de $\mathcal{M}$ qui converge vers $\mu$ telle que

$$(H) \mbox{  :   }  \lim_{n \rightarrow + \infty}
\int \log d(x,I)  d \mu_{n} (x) = \int \log d(x,I)  d \mu(x) > - \infty.$$

Soit $\delta >0$. On prend $L=  - \int \log d(x,I)  d \mu(x) +1$. Par l'hypothèse (H), il existe $n_0$ tel que pour $n \geq n_0$ on ait 

$$\int \log d(x,I)  d \mu_n(x) \geq -L.$$

On peut donc appliquer le théorème précédent à ce $L$ et aux mesures $\mu_n$ pour $n \geq n_0$: on obtient ainsi l'existence de $m_0$ tel que pour $m \geq m_0$

$$\sup_{  n \geq n_0  } h_{\mu_n, \frac{\rho_m}{2} , loc}(f) \leq \delta.$$

Dans la suite on fixe $m \geq m_0$ et on considère $n \geq n_0$. Par le théorème \ref{local}, on a 

$$h_{\mu_n}(f) \leq h_{\mu_n}(f, \Pcal) + h_{\mu_n, \frac{\rho_m}{2}, loc}(f) \leq h_{\mu_n}(f, \Pcal) + \delta.$$

Ici, $\Pcal$ est la partition associée à $\frac{\rho_m}{2}$ et quitte à la bouger un peu on peut supposer que $\mu$ ne charge pas ses bords.

L'hypothèse (H) et l'invariance par $f$ des mesures $\mu_n$ et $\mu$ impliquent que 

$$\lim_{n \rightarrow + \infty} \int \log \frac{\rho_m (x)}{2} d \mu_n(x) = \int \log \frac{\rho_m (x)}{2} d \mu(x) > - \infty.$$

En particulier, par la proposition \ref{semicontinue}, on a 

$$\limsup_{n \rightarrow + \infty} h_{\mu_n}(f) \leq  \limsup_{n \rightarrow + \infty}  h_{\mu_n}(f, \Pcal) + \delta \leq h_{\mu}(f, \Pcal) + \delta \leq  h_{\mu}(f) + \delta.$$

Cela démontre le théorème.

\end{proof}

\bigskip

\bigskip\noindent
Henry De Thélin, Université Paris 13, Sorbonne Paris Cité, LAGA, CNRS (UMR 7539), F-93430, Villetaneuse, France.\\
 {\tt dethelin@math.univ-paris13.fr}

\end{document}